\documentclass[12pt]{amsart}
\usepackage[utf8]{inputenc}
\usepackage[margin=1in]{geometry}

\usepackage[
    backend=biber,
    style=ieee,
    citestyle=numeric-comp,
    sorting=nyt,
    dashed=false,
  ]{biblatex}
\usepackage{hyperref}
\usepackage[hyphenbreaks]{breakurl}
\usepackage{xurl}
\hypersetup{
    colorlinks=true,
    linkcolor=blue,
    citecolor=red,
    breaklinks=true
}

\usepackage{graphicx}
\usepackage{amsmath}
\usepackage{amssymb}
\usepackage{appendix}
\usepackage{amsthm}
\usepackage{booktabs}
\usepackage{graphicx}
\usepackage{subcaption}
\usepackage{comment}
\usepackage{enumerate}
\usepackage{esint}

\usepackage[dvipsnames]{xcolor}

\newtheorem{theorem}{Theorem}[section]
\newtheorem{theorem*}{Theorem}

\newtheorem{prop}[theorem]{Proposition}

\theoremstyle{definition}
\theoremstyle{remark}

\usepackage{amsmath}
\usepackage{amssymb}
\usepackage{mathrsfs}

\newcommand{\R}{{\mathbb R}}

\newcommand{\F}{{\mathcal{F}}}

\newcommand{\cS}{{\mathcal{S}}}

\newcommand{\sS}{{\mathscr{S}}}

\newcommand{\tr}{{\text{tr}}}

\usepackage{todonotes}


\title[Weyl quantisation of Paley-Wiener functions]{A remark on the Weyl quantisation of Paley-Wiener functions}

\author{Helge J. Samuelsen}
\address{Department of Mathematical Sciences,
         Norwegian University of Science and Technology,
         Trondheim, Norway}

\email{helge.j.samuelsen@ntnu.no}
\date{\today}

\begin{document}

\begin{abstract}
We present a short proof of the fact that the Weyl quantisation of a tempered distribution with compactly supported Fourier transform is in the Schatten $p$-class if and only if the symbol is $L^p$-integrable. The proof is based on Werner-Young's inequality from quantum harmonic analysis and a quantum version of Wiener's division lemma.
\end{abstract}
\maketitle

\section{Introduction}
The theory of quantum harmonic analysis was introduced by Werner, and allows for Fourier analysis of operators \cite{Werner}. In this setting, the appropriate replacement of the classical $L^p$-spaces are the Schatten $p$-classes of compact operators $T$ on $L^2(\R^d)$ such that
\[
\|T\|_{\cS^p}=\left(\tr(|T|^p)\right)^\frac{1}{p}<\infty.
\]
Several classical results have been extended to the setting of quantum harmonic analysis, such as Wiener's Tauberian theorem \cite{Luef_Skrettingland_21, Fulsche_Luef_Werner_24}, Fourier restriction \cite{Luef_Samuelsen_24, Muller} and decoupling \cite{Samuelsen_24}. 

From the inherent symplectic structure on $\R^{2d}$, it is convenient to introduce the symplectic Fourier transform when working with quantum harmonic analysis.  Let $\sigma$ denote the standard symplectic form on $\R^{2d}$, and define the symplectic Fourier transform by
\[
\F_\sigma(\Psi)(\zeta)=\int_{\R^{2d}}e^{-2\pi i \sigma(\zeta, z)}\Psi(z)\,dz,
\]
for $\Psi\in \sS(\R^{2d})$. The symplectic Fourier transform extends to an isomorphism on $\sS'(\R^{2d})$ through duality by
\[
\langle \F_\sigma(\tau),\Psi\rangle_{\sS',\sS}=\langle \tau,\F_\sigma{\Psi}\rangle_{\sS',\sS},
\]
for $\tau\in \sS'(\R^{2d})$ and all $\Psi\in \sS(\R^{2d})$.

The theory of quantum harmonic analysis is based on the isometric isomorphism between $L^2(\R^{2d})$ and the class of Hilbert-Schmidt operators $\cS^2$ for the Weyl quantisation \cite{Pool}. For a tempered distribution $\tau\in \sS'(\R^{2d})$, the Weyl quantisation is defined through a sesquilinear dual pairing as the bounded linear operator $L_\tau:\sS(\R^d)\to \sS'(\R^{d})$ such that
\[
\langle L_\tau \varphi,\psi\rangle_{\sS',\sS}=\langle \tau, \mathcal{W}(\psi,\varphi)\rangle_{\sS',\sS},
\]
for every $\varphi,\psi\in \sS(\R^{d})$. The tempered distribution $\tau$ is called the symbol of the operator $L_\tau$, and $\mathcal{W}(\psi,\varphi)$ is the cross-Wigner distribution defined as
\[
\mathcal{W}(\psi,\varphi)(x,\xi)=\int_{\R^d}\varphi\left(x+\frac{t}{2}\right)\overline{\psi\left(x-\frac{t}{2}\right)}e^{-2\pi i \xi\cdot t}\,dt.
\]

Unlike the case $p=2$, there is generally no equivalence between $L^p$-integrability of the symbol and Schatten properties of the Weyl quantisation. Nevertheless, there is an equivalence whenever the symbol has compact Fourier support.
\begin{theorem*}\label{thm:MainThm}
Let $\tau\in \sS'(\R^{2d})$, and assume that $\F_\sigma(\tau)$ is a compactly supported distribution on $\R^{2d}$. Then $L_\tau$ is a compact operator on $L^2(\R^{2d})$ if and only if $\tau\in C_0(\R^{2d})$, and $L_\tau\in \cS^p$ if and only if $\tau\in L^p(\R^{2d})$ for $1\leq p\leq\infty$.
\end{theorem*}
Distributions with compact Fourier support are necessarily real analytic functions on $\R^{2d}$ by \cite[Thm. $7.3.1$]{Hormander-ALPDO1}, and thus Theorem \ref{thm:MainThm} says that the Weyl quantisation is compact if and only if the symbol vanishes at infinity.

A version of Theorem \ref{thm:MainThm} first surfaced in a paper on Fourier restriction in  quantum harmonic analysis \cite{Luef_Samuelsen_24}. Here it was proved under the assumption that $\F_\sigma(\tau)$ is a compactly supported Radon measure. The theorem was later extended to tempered distributions with compact Fourier support by Mishra and Vemuri \cite{Mishra_comp}, and independently by M\"{u}ller \cite{Muller}.

The aim of this paper is to present a short proof of Theorem \ref{thm:MainThm} using techniques from quantum harmonic analysis. For a detailed exposition of quantum harmonic analysis, we refer to \cite{Luef_Samuelsen_24, Luef_Skrettingland_19, Luef_Skrettingland_21, Samuelsen_24,Werner}.

\section{Proof of Theorem \ref{thm:MainThm}}
The main tool from quantum harmonic analysis needed for the proof of Theorem \ref{thm:MainThm} is Werner's operator convolution given by
\begin{equation}\label{eq:DefConv}
T\star S(z)=\tr(T\alpha_z(PSP)),\qquad z\in\R^{2d}.
\end{equation}
Here $Pf(t)=f(-t)$ denotes the parity operator and $\alpha_z$ the operator translation
\[
\alpha_z(T)=\rho(z)T\rho(-z).
\]
The operator $\rho:\R^{2d}\to \mathcal{L}(L^2(\R^d))$ is the symmetric time-frequency shift given by
\[
\rho(x,\xi)f(t)=e^{-\pi i x\cdot \xi}e^{2\pi i \xi\cdot t}f(t-x).
\]
An important property of the operator convolution is that it coincides with the convolution of the Weyl symbols. More specifically, for $\tau\in \sS'(\R^{2d})$ and $\Phi\in \sS(\R^{2d})$ we have
\begin{equation}\label{eq:ConvEquiv}
L_\tau\star L_{\Phi}=\tau*\Phi,
\end{equation}
as they both have the same symplectic Fourier transform \cite{Luef_Skrettingland_19}. 

The proof of Theorem \ref{thm:MainThm} relies on two theorems related to the operator convolution defined by \eqref{eq:DefConv}. The first is a version of Young's inequality for the operator convolution \cite[Thm $3.3$]{Werner}.
\begin{prop}[Werner-Young's inequality]\label{Young op-op}
Suppose $S\in \cS^p$ and $T\in \cS^q$ with $1+r^{-1}=p^{-1}+q^{-1}$.
Then $S\star T\in L^r(\R^{2d})$ and
\begin{align*}
    \|S\star T\|_{L^{r}}\leq \|S\|_{\cS^p}\|T\|_{\cS^q}.
\end{align*}
\end{prop}
The final ingredient of the proof is the following result, first used to derive an equivalence between classical and quantum decoupling.
\begin{theorem}[\cite{Samuelsen_24}, Cor. $3.2.1$]\label{thm:SamuelsenThm}
Let $\Omega\subseteq \R^{2d}$ be a bounded set, and let $1\leq p\leq \infty$. Then there exist $L^2$-normalised $g,h\in \sS(\R^{2d})$ and $C=C(\Omega)>0$ such that if $\tau\in \sS'(\R^{2d})$ and $\F_\sigma(\tau)$ is supported on $\Omega$, then
\[
\|L_\tau\|_{\cS^p}\leq C(\Omega)\|L_\tau\star (g\otimes h)\|_{L^p}.
\]
\end{theorem}
\begin{proof}[Proof of Theorem \ref{thm:MainThm}]
Let $\tau\in \sS'(\R^{2d})$. If $\F_\sigma(\tau)$ is compactly supported, then there exists a bounded subset $\Omega\subseteq\R^{2d}$ such that $\text{supp }(\F_\sigma(\tau))\subseteq \Omega$. 

Assume first that $L_\tau\in \cS^p$. If $\Psi\in C^\infty_c(\R^{2d})$ is a smooth cut-off function equal to $1$ on $\Omega$, then
\[
\F_\sigma(\tau)=\F_\sigma(\tau)\Psi,
\]
and thus $\tau=\tau*\F_\sigma(\Psi)$. By \eqref{eq:ConvEquiv}, and Theorem \ref{Young op-op}, it follows that
\begin{equation}\label{eq:TempDistBound}
\|\tau\|_{L^p}=\|\tau*\F_\sigma(\Psi)\|_{L^p}=\|L_\tau\star L_{\F_\sigma(\Psi)}\|_{L^p}\leq \|L_{\F_\sigma(\Psi)}\|_{\cS^1}\|L_\tau\|_{\cS^p}.
\end{equation}
Since the Weyl quantisation of a Schwartz function belongs to $\cS^1$ by \cite[Thm. C$.16$]{Zworski}, it follows that $\tau\in L^p(\R^{2d})$.

Let $\tau\in L^p(\R^{2d})$. Then by \eqref{eq:ConvEquiv} and Theorem \ref{thm:SamuelsenThm}, it follows that
\begin{equation}\label{eq:WeylBound}
\|L_{\tau}\|_{\cS^p}\leq C(\Omega)\|L_\tau\star(g\otimes h)\|_{L^p}=C(\Omega)\|\tau* \mathcal{W}(g,h)\|_{L^p}\leq C(\Omega)\|\mathcal{W}(g,h)\|_{L^1}\|\tau\|_{L^p},
\end{equation}
from the classical Young's inequality for convolutions. Since the cross-Wigner distribution of two Schwartz functions is again a Schwartz function \cite[Thm. $11.2.5$]{Grochenig}, it follows that $L_\tau\in \cS^p$.

Assume now that $\tau\in C_0(\R^{2d})$. Since $C_0(\R^{2d})$ is the $L^\infty$-closure of $\sS(\R^{2d})$, there exists a sequence $\{\varphi_n\}\subseteq \sS(\R^{2d})$ converging uniformly to $\tau$. By \eqref{eq:WeylBound}, the sequence of compact operators $\{L_{\varphi_n}\}$ converges to $L_\tau$ in the norm topology on $\mathcal{L}(L^2(\R^d))$ which implies that $L_\tau$ is a compact operator.

If $L_\tau$ is a compact operator, then there exists a sequence of Schwartz operators $\{L_{\varphi_n}\}$ converging to $L_\tau$ in the norm topology of $\mathcal{L}(L^2(\R^d))$ by the density of Schwartz operators in the finite rank operators \cite{Keyl-Kiukas-Werner_16}. By \eqref{eq:TempDistBound}, it follows that $\varphi_{n}$ converges uniformly to $\tau$. This implies that $\tau\in C_0(\R^{2d})$ as it is the uniform closure of $\sS(\R^{2d})$.
\end{proof}

\section*{Acknowledgement}
The author would like to thank Sigrid Grepstad and Franz Luef for valuable feedback on earlier versions of the manuscript.

\printbibliography

\end{document}